\documentclass{article}
\usepackage[utf8]{inputenc}
\usepackage{amsmath,amsfonts,amssymb,amsthm}
\usepackage{bbm}
\usepackage[margin=1in]{geometry}
\usepackage{color}
\usepackage{tikz}
\usepackage{tikz-cd}
\usepackage{authblk}
\usepackage{hyperref}
\usepackage{csquotes}
\usepackage[
  hyperref=auto,
  mincrossrefs=999,
  backend=biber,
  style=alphabetic,
  sorting=nyt
]{biblatex}
\newcommand{\norm}[1]{\left\lVert #1 \right\rVert}
\newcommand{\mor}{\text{Mor}}
\newcommand{\Uelt}[2]{U \left( \substack{#1 \\ #2} \right)}
\newcommand{\Velt}[2]{V \left( \substack{#1 \\ #2} \right)}
\newcommand{\Aelt}[3]{A_{#1} \left( \substack{#2 \\ #3} \right)}

\DeclareMathOperator{\onb}{onb}
\DeclareMathOperator{\mult}{\mathfrak{m}}

\DeclareMathOperator{\id}{Id}
\DeclareMathOperator{\op}{op}

\DeclareMathOperator{\tr}{Tr}
\DeclareMathOperator{\del}{\partial}
\DeclareMathOperator{\ovtimes}{\overline{\otimes}}
\DeclareMathOperator{\mtimes}{\otimes_{(M,\del)}}
\DeclareMathOperator{\rep}{Rep}
\DeclareMathOperator{\End}{End}

\newcommand{\vect}[1]{\mathfrak{v} \left( #1 \right)}

\makeatletter
\newtheorem*{rep@theorem}{\rep@title}
\newcommand{\newreptheorem}[2]{%
\newenvironment{rep#1}[1]{%
 \def\rep@title{#2 \ref{##1}}%
 \begin{rep@theorem}}%
 {\end{rep@theorem}}}
\makeatother

\theoremstyle{definition}
\newtheorem{definition}[subsection]{Definition}
\newtheorem{definition/proposition}[subsection]{Definition/Proposition}
\newtheorem{theorem}[subsection]{Theorem}
\newreptheorem{theorem}{Theorem}
\newtheorem{lemma}[subsection]{Lemma}
\newtheorem{proposition}[subsection]{Proposition}
\newtheorem{corollary}[subsection]{Corollary}
\newtheorem{remark}[subsection]{Remark}

\newtheorem{example}[subsection]{Example}
\newtheorem{conjecture}[subsection]{Conjecture}

\title{Equivariant representation theory for proper actions on discrete spaces}
\author[1,2]{Lukas Rollier \thanks{email: lukasrollier@live.be} \thanks{Partly supported by FWO research project G090420N of the Research Foundation Flanders and by FWO-PAS
research project VS02619N}}
\date{}

\begin{document}

\maketitle

\begin{abstract}
    Starting from any proper action of any locally compact quantum group on any discrete quantum space, we show that its equivariant representation theory yields a concrete unitary $2$-category of finite type Hilbert bimodules over the discrete quantum space, from which the quantum group and its action may be completely reconstructed as in \cite{Rollier2025}. In particular, this shows that any locally compact quantum group acting properly on a discrete quantum space must be an algebraic quantum group.
\end{abstract}

\section{Introduction} \label{sc: introduction}

In the theory of compact quantum groups, special interest is given to their representation categories. As shown by Woronowicz \cite{Woronowicz88}, similarly to the case of classical compact groups, every representation of a compact quantum group is equivalent to a direct sum of irreducible, finite dimensional, unitary representations. Moreover, every representation has a categorical conjugate, which goes to show that a lot of the algebraic properties of the compact quantum group are reflected in its representation category. The content of Woronowicz' Tannaka-Krein duality is then precisely that the only information missing from the representation category is contained in its canonical fiber functor. 

For more general locally compact quantum groups, the situation is not as convenient. The central tool which allows for the detailed study of the representation category in the compact setting is the Haar state, which does not exist in the general locally compact setting. Nonetheless, there are always Haar weights, so one could hope to exploit these for a similar study. Still, it is well-known that the existence of a categorical conjugate to a representation of a locally compact quantum group forces the representation to be finite dimensional, and finite dimensional representations can only ever remember the Bohr compactification of a locally compact quantum group. 

Expanding our attention not just to locally compact quantum groups, but their actions as well, more is possible. In \cite{DeCommerDeRo2024}, De Commer and De Ro introduced the equivariant representation category of an action of a locally compact quantum group on a von Neumann algebra. This notion was then used to study approximation properties of such actions. Importantly, the notion of an equivariant representation category is a true generalisation of the ordinary representation category. Indeed, every representation of a locally compact quantum group is equivariant with respect to the trivial action on $\mathbb{C}$.

Recently in \cite{Rollier2025}, the author showed that for any discrete quantum space $(M,\del)$, as in definition \ref{def: discrete quantum space}, and a sufficiently well-behaved unitary $2$-category of finite type Hilbert-$M$-bimodules $\mathcal{C}$, it is possible to construct an algebraic quantum group $\mathbb{G}$ which admits an action $\alpha$ on $M$ preserving $\del$, such that the $\alpha$-equivariant representation theory of $\mathbb{G}$ is precisely given by $\mathcal{C}$. Hence, at least some actions can be reconstructed from their equivariant representation theory. The natural question then arises: which ones?

All of the actions resulting from the reconstruction in \cite{Rollier2025} were proper in the sense of definition \ref{def: generators of faithful part + proper action}. This properness may be seen as an equivariant version of compactness, and in light of the non-equivariant setting, it is then natural to restrict our attention to this class of actions. The current article shows that properness of an action on a discrete quantum space is sufficient in order for its representation theory to abide by the conditions of the reconstruction in \cite{Rollier2025}. In particular, given such an action of any locally compact quantum group $\mathbb{G}$, we are able to find a canonical strongly dense $*$-subalgebra $\mathcal{O}(\mathbb{G}) \subset L^{\infty}(\mathbb{G})$ which lies in the domain of both Haar weights, and such that the restriction of the coproduct to $\mathcal{O}(\mathbb{G})$ turns it into an algebraic quantum group in the sense of \cite{VanDaele98, KustermansVanDaele97}.

In \cite{LandstadVanDaele2008}, Landstad and Van Daele show that a classical locally compact group has a description as an algebraic quantum group if and only if it contains a compact open subgroup. In this case, the canonical action of the group on cosets of this compact open subgroup by translation is a proper action on a discrete space, and conversely for any proper action on a discrete space of a locally compact group, the stabiliser of any point yields a compact open subgroup. The question now arises whether this property lifts to all locally compact quantum groups, and it is tempting to believe the following conjecture.

\begin{conjecture}
    For any locally compact quantum group $\mathbb{G}$, the following are equivalent.
    \begin{enumerate}
        \item There exists a discrete quantum space $(M,\del)$ in the sense of definition \ref{def: discrete quantum space} and a proper $\del$-preserving right action $\alpha$ of $\mathbb{G}$ on $M$ in the sense of definition \ref{def: generators of faithful part + proper action}.
        
        \item There exists a strongly dense $*$-subalgebra $\mathcal{O}(\mathbb{G}) \subset L^{\infty}(\mathbb{G})$ which lies in the domain of the Haar weights, such that the restriction of the coproduct to $\mathcal{O}(\mathbb{G})$ yields an algebraic quantum group in the sense of \cite{VanDaele98, KustermansVanDaele97}.
    \end{enumerate}
\end{conjecture}

Note that the implication $1. \Rightarrow 2.$ is shown in this paper, so the conjecture only concerns the reverse direction. Also, the direction $2. \Rightarrow 1.$ holds for all classical groups by \cite{LandstadVanDaele2008}, and for all compact and discrete quantum groups. Indeed, when $\mathbb{G}$ is compact, one can take the trivial action on $\mathbb{C}$, and when $\mathbb{G}$ is discrete, one can take the canonical action on $(L^{\infty}(\mathbb{G}),\psi)$ given by the coproduct, where $\psi$ is the right Haar weight, which can always be normalised to be a delta-form.

The structure of this article is as follows. In section \ref{sc: preliminaries}, we recall some preliminaries on discrete quantum spaces and locally compact quantum group actions on them. In section \ref{sc: properties of the action}, we define what it means for such an action to be proper, and we prove some nice properties that proper actions on discrete spaces enjoy. In section \ref{sc: equivariant representation theory}, we exploit these nice properties to get a description of the equivariant representation theory of such a proper action on a discrete space. This equivariant representation theory admits properties which directly generalise those of the representation theory of compact quantum groups. Finally, in section \ref{sc: algebraic quantum group}, we use this equivariant representation theory to define concretely a strongly dense $*$-subalgebra $\mathcal{O}(\mathbb{G}) \subset L^{\infty}(\mathbb{G})$ such that $(\mathcal{O}(\mathbb{G}), \Delta)$ is an algebraic quantum group. This then allows to strengthen the results from \cite{Rollier2025}.

\subsection*{Acknowledgment}
The author thanks Stefaan Vaes and Kenny De Commer for fruitful discussions, and critical insights. We also thank the Research Foundation Flanders, which partly funded this research by FWO research project G090420N and by FWO-PAS research project
VS02619N.

\section{Preliminaries} \label{sc: preliminaries}

As this article reports on the reverse direction of the reconstruction devised in \cite{Rollier2025}, it can help to read that paper for the general understanding of this one. However, we will not use any of its results until section \ref{sc: algebraic quantum group}. We expect the reader to be familiar with the general theory of locally compact quantum groups as initiated in \cite{KustermansVaes99,KustermansVaes03}. Furthermore, we briefly provide some preliminaries on discrete quantum spaces and actions of locally compact quantum groups on them. These preliminaries are partly taken from \cite{Rollier2025}.

We fix for the entire article a type I von Neumann algebra $M := \ell^{\infty} \bigoplus_{i \in I} M_i$, where $I$ is some index set, and $M_i \cong \mathbb{C}^{d_i \times d_i}$ for some integers $d_i \in \mathbb{N}$. We denote the algebra of finitely supported elements by $M_0 \subset M$, i.e. for every $x \in M_0$, there is some finite set $F \subset I$ such that $x \in \bigoplus_{i \in F} M_i$. Recall that we can write $M_0 \subset M \subset \mathcal{M}(M_0)$, where $\mathcal{M}(M_0)$ is the multiplier algebra, consisting of all possibly unbounded elements of $\prod_{i \in I} M_i$.

In the spirit of noncommutative geometry, we will view $M$ as the algebra of bounded functions on some noncommutative discrete space. The following will then play the role of the counting measure on that space. This was also defined in \cite[Definition 2.1.1]{Rollier2025}.

\begin{definition} \label{def: delta-form}
     We say an nsf weight $\del: M^+ \to [0,\infty]$ is a delta-form if the adjoint of the multiplication $M \otimes M \to M$ with respect to the GNS inner product coming from $\del$ is an isometry. This is equivalent to the existence of a positive invertible element $\sigma = \mathcal{M}(M_0)$, such that $\tr(1_i \sigma^{-1}) = d_i$ and $\del$ restricts to every $M_i$ as $\del_{| M_i}: M_i \to \mathbb{C}: x \mapsto d_i \tr(x \sigma)$.
\end{definition}

\begin{example} \label{ex: markov trace is delta-form}
    Given any direct sum of matrix algebras, $M$, the markov trace $\tr_M: E^i_{k,l} \mapsto \delta_{k,l} d_i$ is the unique tracial delta-form. In particular, when $M = \ell^{\infty}(I)$, the unique delta-form is given by integration against the counting measure.
\end{example}

Fix now for the rest of the article a delta-form $\del$ on $M$. We denote as in definition \ref{def: delta-form} by $\sigma \in \mathcal{M}(M_0)$ the multiplier such that $\del(x) = d_i \tr(x \sigma)$ for any $x \in M_i$. We define a one parameter group of algebra automorphisms\footnote{note that these are only $*$-automorphisms when they are trivial} $\mu^t: M_0 \to M_0: x \mapsto \sigma^{-t} x \sigma^t$. Clearly, this extends to all of $\mathcal{M}(M_0)$. Note however that in general, this extension does not restrict to $M$ itself.

As in \cite[Definition 2.1.3]{Rollier2025}, we now define discrete quantum spaces.
\begin{definition} \label{def: discrete quantum space}
    A discrete quantum space is a pair $(M,\del)$ consisting of a direct sum of matrix algebras $M$, and a delta-form $\del$ on $M$. For our fixed discrete quantum space $(M,\del)$, we denote by $\mathcal{F}_n$ the pre-Hilbert space $M_0^{\otimes n+1}$ with inner product given by
    \[
    \langle \vect{x_0 \otimes \cdots \otimes x_n} , \vect{y_0 \otimes \cdots \otimes y_n} \rangle = \del(y_0^* x_0) \cdots \del(y_n^* x_n).
    \]
    Here we write $\vect{x_0 \otimes \cdots \otimes x_n} \in \mathcal{F}_n$ instead of just $x_0 \otimes \cdots \otimes x_n$ to specify that we mean a vector in a pre-Hilbert space more than an element of an algebra. Then we denote by\footnote{The spaces $L^2(M^{n+1})$ depend, of course, also on $\del$, but we supress this in the notation to keep it lighter.} $L^2(M^{n+1})$ the closure of $\mathcal{F}_n$ with respect to this inner product. Now, as was mentioned in definition \ref{def: delta-form}, $\mult: \mathcal{F}_1 \to \mathcal{F}_0: \vect{x \otimes y} \mapsto \vect{xy}$ extends to a bounded linear map whose adjoint is an isometry.
\end{definition}

One easily checks that there is an easy orthonormal basis of $L^2(M)$ given by
\begin{equation}
    \left\{ \left. d_i^{-1/2} \vect{E^i_{k,l} \sigma^{-1/2}} \right| i \in I \text{ and } k,l \in \{1, \ldots, i\} \right\}. \label{eq: orthonormal basis of F_0}
\end{equation}
Throughout the article, we will often sum over $\onb(\mathcal{F}_0)$, by which we mean any orthonormal basis of $L^2(M)$, which as a set is contained in $\mathcal{F}_0$. We will make sure that the choice of orthonormal basis is of no consequence, and hence the one in \eqref{eq: orthonormal basis of F_0} may always be chosen.

Next, we must speak of bimodules over our discrete quantum space, as in \cite[Definition 2.1.5]{Rollier2025}.

\begin{definition} \label{def: Hilbert-M-bimodule}
    Let $H$ be any Hilbert space, which is equipped with $\lambda_H: M \to B(H)$ and $\rho_H: M^{\op} \to B(H)$ two normal unital $*$-homomorphisms such that $\lambda_H(x) \rho_H(y^{\op}) = \rho_H(y^{\op}) \lambda_H(x)$ for any $x,y \in M$. Then we say $(H,\lambda_H,\rho_H)$ is a Hilbert-$M$-bimodule. We will often just refer to $H$ as the bimodule, and implicitly assume $\lambda_H,\rho_H$. We will denote, for any $\xi \in H$ and $x,y \in M$
    \[
    x \cdot \xi \cdot y :=  \lambda_H(x) \rho_H(y^{\op}) (\xi).
    \]
    If it so happens that for any $x \in M_0$ we get that $\lambda_H(x)$ and $\rho_H(x^{\op})$ are finite rank operators, we say that the Hilbert-$M$-bimodule $H$ is of finite type.
\end{definition}

\begin{example} \label{ex: bimodule structure on F_n}
    The spaces $\mathcal{F}_n$ from definition \ref{def: discrete quantum space} come with a canonical $M$-bimodule structure given by
    \[
    a \cdot \vect{x_0 \otimes \cdots \otimes x_n} \cdot b := \vect{ax_0 \otimes \cdots \otimes x_n \mu^{-1/2}(b)}.
    \]
    One easily checks that this bimodule structure extends to $L^2(M^{n+1})$, making $L^2(M^{n+1})$ a Hilbert-$M$-bimodule for any $n \in \mathbb{N}$. Note that when $n \neq 0$ and $M$ is infinite-dimensional, $L^2(M^{n+1})$ is not of finite type.
\end{example}

\begin{definition} \label{def: relative tensor product}
    In \cite{Connes1980}, Connes defined the relative tensor product of two Hilbert-bimodules $H,K$ over a von Neumann algebra with respect to some weight. In the special case where we take this von Neumann algebra to be $M$ and this weight to be our delta-form, this becomes
    \[
    H \mtimes K := \{ \xi \in H \otimes K | \forall x \in M_0:  (\rho_H(x^{\op}) \otimes 1) \xi = (1 \otimes \lambda_K(\mu^{1/2}(x))) \xi \},
    \]
    which comes equipped with an $M$-bimodule structure $\lambda_{H \mtimes K} := \lambda_H \otimes 1$, $\rho_{H \mtimes K} = 1 \otimes \rho_K$. Note that the relative tensor product of finite type Hilbert-$M$-bimodules is again a finite type Hilbert-$M$-bimodule.
\end{definition}
Recall from \cite[Proposition 2.1.11]{Rollier2025} that the orthogonal projection $H \otimes K \to H \mtimes K$ is given by
\[
P := \sum_{x \in \onb(\mathcal{F}_0)} \rho_H(y^{\op}) \otimes \lambda_K(\mu^{1/2}(x^*)).
\]
We also get from \cite[Proposition 2.1.12]{Rollier2025} that for any Hilbert-$M$-bimodule $H$, there are canonical unitary $M$-bimodular isomorphisms $L^2(M) \mtimes H \cong H \cong H \mtimes L^2(M)$ given by the restrictions of the maps
\begin{align*}
    &H \otimes L^2(M) \to H: \xi \otimes \vect{x} \mapsto \xi \cdot \mu^{1/2}(x) \\
    &L^2(M) \otimes H \to H: \vect{x} \otimes \xi \mapsto x \cdot \xi.
\end{align*}

We finally define right actions on our discrete quantum space. For general actions of locally compact quantum groups, see \cite{Vaes01}.

\begin{definition} \label{def: right action}
    A right action $\alpha$ of a locally compact quantum group $\mathbb{G}$ on the discrete quantum space $(M,\del)$ is a normal, unital, injective $*$-homomorphism $\alpha: M \to M \ovtimes L^{\infty}(\mathbb{G})$ satisfying the action property
    \[
    (\id \ovtimes \Delta) \circ \alpha = (\alpha \ovtimes \id) \circ \alpha,
    \]
    where $\Delta$ denotes the coproduct of $\mathbb{G}$, and which moreover preserves the weight $\del$ in the sense that for any $x$ in the domain of $\del$, and any $\omega \in L^{\infty}(\mathbb{G})_*$, the predual of $L^{\infty}(\mathbb{G})$, we have that $(\id \otimes \omega) \alpha(x)$ lies in  the domain of $\del$ and 
    \[
    \del(\id \otimes \omega)\alpha(x) = \del(x) \omega(1).
    \]
\end{definition}

\section{Properties of the action and the orbit equivalence relation} \label{sc: properties of the action}

We fix for the entirety of this article a locally compact quantum group $\mathbb{G}$ which admits a right action $\alpha$ on a direct sum of matrix algebras $M = \bigoplus_{i \in I} M_i$ such that $\alpha$ preserves a delta-form $\del$. Recall that $\sigma$ denotes the weight of $\del$ with respect to the Markov trace, and that $\mu^t$ denotes conjugation by $\sigma^{-t}$. We denote by $\varphi$ and $\psi$ the left and right Haar weights of $\mathbb{G}$ respectively, normalised such that $\psi = \varphi \circ R$, where $R$ denotes the unitary antipode of $\mathbb{G}$. We also denote by $\alpha^{\op}$ the opposite action on $M^{\op}$, satisfying $\alpha^{\op}(x^{\op}) = (\op \otimes R) \alpha(x)$, where $\op$ denotes the canonical anti-isomorphism $M \to M^{\op}: x \mapsto x^{\op}$.

\begin{definition} \label{def: generators of faithful part + proper action}
    Denoting for any $x,y \in M_0$ the following
    \[
    u_{y,x} := (\del(y^* \cdot ) \otimes \id)\alpha(x) \in L^{\infty}(\mathbb{G})
    \]
    such that
    \[
    \alpha(x) = \sum_{y \in \onb(\mathcal{F}_0)}y \otimes u_{y,x}
    \]
    with convergence in the strong operator topology, we say the action $\alpha$ is proper if the following two conditions hold for any $x,y \in M_0$.
    \begin{itemize}
        \item $u_{x,y} \in C_r(\mathbb{G})$
        \item $u_{x,y}$ lies in the domain of $\varphi$. Here, by the domain of $\varphi$, we mean the linear span of positive elements which have bounded image under $\varphi$.
    \end{itemize}
    Note that for any $i,j \in I$, $u_{1_i,1_j}$ is a positive element of $L^{\infty}(\mathbb{G})$. 
\end{definition}

The reader may note that in the special case where $\mathbb{G}$ is a classical locally compact group, and $M = \ell^{\infty}(I)$ denotes the bounded functions on a classical discrete space, both of the properties defining properness in the sense of definition \ref{def: generators of faithful part + proper action} are equivalent with properness in the ordinary sense of the literature. Indeed, the first implies immediately that the stabiliser of every point in $I$ is compact. The second implies that the stabiliser of any point has a bounded Haar measure. Since these stabilisers are clopen subgroups, this forces them to again be compact, and hence the action is proper in the sense which is more familiar in the literature, justifying our nomenclature. It is unfortunately unclear whether either of these properties implies the other in the quantum setting.

\begin{proposition} \label{prop: calculation on u_(x,y)}
    The following are valid rules of calculation for the elements $u_{x,y} \in L^{\infty}(\mathbb{G})$.
    \begin{enumerate}
        \item For any $x,y \in M_0$, we have $u_{x,y}^* = u_{\mu(x)^*,y^*}$.

        \item For any $x \in M_0$, we have $\sum_{i \in I} u_{1_i,x} = \del(x) 1 = \sum_{i \in I} u_{x^*,1_i}$, with convergence in the strong operator topology.

        \item For any $x,y,z \in M_0$, we have $\sum_{t \in \onb(\mathcal{F}_0)} u_{x,yt} u_{z,t^*} = u_{xz,y} = \sum_{t \in \onb(\mathcal{F}_0)} u_{x,t}u_{z,t^*y}$ and similarly on the other side $\sum_{t \in \onb(\mathcal{F}_0)} u_{xt,y} u_{t^*,z} = u_{x,yz} = \sum_{t \in \onb(\mathcal{F}_0)} u_{t,y}u_{t^*x,z}$.

        \item For any $x,y \in M_0$, we have $\Delta(u_{x,y}) = \sum_{z \in \onb(\mathcal{F}_0)} u_{x,z} \otimes u_{z,y}$.
    \end{enumerate}
\end{proposition}
\begin{proof}
    \begin{enumerate}
        \item One calculates as follows.
        \begin{align*}
            u_{x,y}^* = ((\del(x^* \cdot) \otimes \id)\alpha(y))^* = (\del(\cdot x) \otimes \id) \alpha(y^*) = (\del(\mu(x) \cdot ) \otimes \id) \alpha(y^*) = u_{\mu(x)^*,y^*}
        \end{align*}

        \item Suppose for the first equality that $x$ is positive. Then by normality of $\del$, we get for any positive $\omega \in L^{\infty}(\mathbb{G})_*$ that $\omega \left( \sum_{i \in I} u_{1_i,x} \right) = \sum_{i \in I} (\del(1_i \cdot) \otimes \omega) \alpha(x) = \sup_{\substack{F \subset I \\ |F| < \infty }} \del \left((1_F \otimes \omega) \alpha(x) \right) = \del\left(( \id \otimes \omega)\alpha(x)\right) = \del(x) \omega(1)$. As the predual seperates points, we have our statement for positive $x$. For general $x$, the statement now follows from linearity.

        For the second equality, it suffices to use normality of $\alpha$.

        \item Note first that the sums are finite since we take $x,y,z,t \in M_0$. Note also that for $x \in M_0$, we have $\mult^*(\vect{x}) = \sum_{t \in \onb(\mathcal{F}_0)} \vect{t} \otimes \vect{t^*x} = \sum_{t \in \onb(\mathcal{F}_0)} \vect{xt} \otimes \vect{t^*}$. Now, to show the second set of equalities, we calculate as follows.
        \begin{align*}
            u_{x,yz} = (\del(x^* \cdot) \otimes \id)\alpha(yz) &= (\del(x^* \cdot) \otimes \id)\alpha(y) \alpha(z) \\
            &= \sum_{t \in \onb(\mathcal{F}_0)} (\del(t^* \cdot) \otimes \del(x^*t \cdot) \otimes \id) \alpha(y)_{13} \alpha(z)_{23} \\
            &= \sum_{t \in \onb(\mathcal{F}_0)} u_{t,y} u_{t^*x,z}
        \end{align*}
        The first set of equalities now follows from the fact that we have now shown $\mult$ to be an intertwiner of $(U_0)_{13}(U_0)_{23}$ to $U_0$, where $U_0$ is the unitary implementation of $\alpha$ as in definition \ref{def: unitary implementation} below. It then follows that $\mult^*$ is an intertwiner in the other direction.

        \item We have the following.
        \begin{align*}
            &\Delta(u_{x,y}) = (\del(x^* \cdot ) \otimes \Delta)\alpha(y) = (\del(x^* \cdot ) \otimes \id \otimes \id) (\alpha \otimes \id)\alpha(y) \\
            &= \sum_{z \in \onb(\mathcal{F}_0)} (\del(x^* \cdot ) \otimes \id \otimes \id)(\alpha(z) \otimes u_{z,y}) = \sum_{z \in \onb(\mathcal{F}_0)} u_{x,z} \otimes u_{z,y}
        \end{align*}
        
    \end{enumerate}
\end{proof}

We assume for the rest of this section that the fixed action $\alpha$ is proper. Recall from \cite{Vaes01} the unitary implementation of the action $\alpha$.

\begin{definition} \label{def: unitary implementation}
    By \cite{Vaes01}, and the fact that $\alpha$ preserves the delta-form $\del$, we get that there is a unitary representation $U_0 \in B(L^2(M)) \ovtimes L^{\infty}(\mathbb{G})$ satisfying $(\omega_{x,y} \otimes \id)U_0 = u_{x,y}$, where $\omega_{x,y} \in B(L^2(M))_*$ is defined by $\omega_{x,y}(T) = \langle T \vect{y}, \vect{x} \rangle$ for any $x,y \in M_0$.
\end{definition}

\begin{lemma} \label{lem: modular data on generators}
    Denote by $S$ the antipode of $\mathbb{G}$, by $R$ its unitary antipode, and by $\tau$ its scaling group. We get the following properties for any $x,y \in M_0$.
    \begin{enumerate}
        \item $u_{x,y}$ lies in the domain of $S$ and $S(u_{x,y}) = u_{y,x}^* = u_{\mu(y)^*,x^*}$
        \item $\tau_t(u_{x,y}) = u_{\mu^{it}(x),\mu^{it}(y)}$
        \item $R(u_{x,y}) = u_{\mu^{1/2}(y)^*,\mu^{1/2}(x)^*}$
    \end{enumerate}
\end{lemma}
\begin{proof}
\begin{enumerate}
    \item Recall that for any unitary representation $U \in B(H) \ovtimes L^{\infty}(\mathbb{G})$, and any functional $\omega \in B(H)_*$ we have $S((\omega \otimes \id)U) = (\omega \otimes \id)U^*$. Now, since $u_{x,y} = (\omega_{x,y} \otimes \id)U_0$, we get that $S(u_{x,y}) = (\omega_{x,y} \otimes \id)U_0^* = u_{y,x}^* = u_{\mu(y)^*,x^*}$.

    \item By discreteness of $M$ and positivity of $\sigma$, we get that $\mu$ admits a basis of eigenvectors in $M_0$, so by linearity it suffices to suppose that $x$ and $y$ are eigenvectors of $\mu$, say with eigenvalues $\mu_x$ and $\mu_y$. In that case, we find by the previous point that $\tau_{-ni}(u_{x,y}) = S^{2n}(u_{x,y}) = u_{\mu^{-n}(x),\mu^n(y)} = \overline{\mu_x}^{-n} \mu_y^n u_{x,y}$. Define the entire function $f:\mathbb{C} \to L^{\infty}(\mathbb{G}): t \mapsto \overline{\mu_x}^{-it} \mu_y^{-it} \tau_t(u_{x,y})$. Then on the one side, we have for any $s \in \mathbb{R}$ and $t \in \mathbb{C}$ that $\norm{f(t+s)} = \norm{f(t)}$. On the other hand, one easily checks that $f(t+i) = f(t)$ for any $t \in \mathbb{C}$. Consequently, the entire function $f$ is bounded, and hence constant by Liouville's theorem. It follows that $\tau_t(u_{x,y}) = u_{\mu^{it}(x),\mu^{it}(y)}$.
    
    \item This now follows from the fact that we must have $S = R \circ \tau_{-i/2}$ wherever defined.
\end{enumerate}    
\end{proof}

\begin{lemma}
    There is an equivalence relation $\approx$ on $I$ given by $i \approx j$ iff $u_{1_i,1_j} \neq 0$. Moverover, for any $W \subset I$, we have $\alpha(1_W) = 1_W \otimes 1 = \alpha^{\op}(1_W)$ if and only if $W$ is invariant under $\approx$. Due to this property, we will call $\approx$ the orbit equivalence relation of $\alpha$. We will call its equivalence classes the orbits of $\alpha$, and denote the set of these orbits as $\mathcal{E}$.
\end{lemma}
\begin{proof}
    Since $R(u_{1_i,1_j}) = u_{1_j,1_i}$, the relation is symmetric. Note also that
    \[
    u_{\sigma_i^{-1},1_j} = \sum_{\substack{ x \in \onb(1_i \cdot \mathcal{F}_0) \\ y \in \onb(1_j \cdot \mathcal{F}_0) }} u_{x,y} u_{x,y}^*
    \]
    so $u_{\sigma_i^{-1},1_j} = 0$ if and only if $u_{x,y}$ is zero for every $x \in M_i, y \in M_j$. Now, since $\sigma_i^{-1}$ and $1_i$ are bounded above and below by multiples of each other, $u_{\sigma_i^{-1},1_j} = 0$ if and only if $u_{1_i,1_j} = 0$. Finally, if $i \not \approx j$, we get that for every $x \in M_i$, $y \in M_j$, $k \in I$ and $z \in M_k$
    \begin{equation}
    0 = (u_{x_{(1)},z} \otimes 1) \Delta(u_{x_{(2)},y}) = \sum_{t \in \onb(1_k \cdot \mathcal{F}_0)} u_{x,zt} \otimes u_{t,y}, \label{eq: transitivity calculation}
    \end{equation}
    Where we denote $\mult^*(a) := a_{(1)} \otimes a_{(2)}$ as a sort of Sweedler notation. Note that for $a \in M_0$, this amounts to taking a finite sum of pure tensors in $M_0 \otimes M_0$.
    Multiplying the expression \eqref{eq: transitivity calculation} on the right with its adjoint, and then summing over all $y \in \onb(1_j \cdot \mathcal{F}_0)$ yields
    \begin{align*}
        0 &= \sum_{t,s \in \onb(1_k \cdot \mathcal{F}_0)} u_{x,zt}u_{\mu(x)^*,s^*z^*} \otimes u_{t \mu(s)^*,1_j} \\
        &= \sum_{t,s,r \in \onb(1_k \cdot \mathcal{F}_0)} \del(\mu(s)t^*r) u_{x,zt}u_{\mu(x)^*,s^*z^*} \otimes u_{r,1_j} \\
        &= \sum_{r \in \onb(1_k \cdot \mathcal{F}_0)} u_{x\mu(x)^*,zrz^*} \otimes u_{r,1_j}
    \end{align*}
    Now summing over $z \in \onb(1_k \cdot \mathcal{F}_0)$, and taking $x = 1_i$, we find $d_k^{-2} u_{1_i,1_k} \otimes u_{\sigma_k^{-2},1_j} = 0$. Now, since $\sigma_k$ and $1_k$ are bounded above and below by multiples of each other, we get $u_{\sigma_k^{-2},1_j} = 0$ if and only if $u_{1_k,1_j} = 0$. Hence, if $i \not \approx j$, then for any $k \in I$, $i \not \approx k$ or $k \not \approx j$. This shows transitivity. 

    Note now that $\sum_{j \in I} u_{1_i,1_j} = \del(1_i)1 \neq 0$ for any $i \in I$. Hence, there must exist for every $i \in I$ some $j \in I$ such that $i \approx j$. Then by symmetry $j \approx i$, and by transitivity $i \approx i$. Hence, $\approx$ is also reflexive.

    Clearly $\alpha(1_W) = 1_W \otimes 1$ if and only if $\alpha^{\op}(1_W) = 1_W \otimes 1$
    Now, take $W \subset I$ arbitrarily, and suppose that $\alpha(1_W) = 1_W \otimes 1$. This forces
    \begin{align*}
        u_{1_i, 1_W} =& (\del(1_i \cdot) \otimes 1) \alpha(1_W) \\
        =& \del(1_i 1_W) 1 \\
        =& \delta_{i \in W} \del(1_i) \\
        =& \delta_{i \in W} u_{1_i,1} .
    \end{align*}
    Hence, for any $i \in W$, and $j \in I$ with $i \approx j$, we must have $u_{1_i,1_W} = u_{1_i,1} = u_{1_i,1_W} + u_{1_i,1_{I \backslash W}}$, and since these are positive, $u_{1_i,1_{I \backslash W}} = 0$, so $j \in W$. It follows that $W$ is invariant under $\approx$.

    Conversely, suppose $1_W$ is invariant under $\approx$. Then for any $i \in I$, $x \in M_i$, we have $u_{x,1_W} = \delta_{i \in W} \del(x^*)$. It follows that $\alpha(1_W) = 1_W \otimes 1$.
\end{proof}

\begin{lemma} \label{lem: concrete invariance of functionals}
    For any $i \approx k \in I$ and $e \in M_0$, we have $d_i^{-2} \varphi(u_{\sigma_i^{-2},e}) = d_k^{-2} \varphi(u_{\sigma_k^{-2},e})$ and $d_i^{-2}\psi(u_{e,\sigma_i^{-2}}) = d_k^{-2} \psi(u_{e,\sigma_k^{-2}})$.
\end{lemma}
\begin{proof}
    Note that for any $e \in M_0$, we have the following\footnote{The steps in this calculation were made by choosing an explicit orthonormal basis of $\mathcal{F}_0$. This, however, inhibits swift readability so we must invite the reader to do the same.}, using again the Sweedler type notation $\mult^*(a) = a_{(1)} \otimes a_{(2)}$.
    \begin{align*}
        &(u_{(1_i)_{(1)},(1_k)_{(1)}} \otimes 1) \Delta(u_{(1_i)_{(2)},e})(u_{(1_i)_{(3)},(1_k)_{(2)}} \otimes 1) \\
        &= \sum_{t \in \onb(\mathcal{F}_0)} u_{(1_i)_{(1)},(1_k)_{(1)}}u_{(1_i)_{(2)},t}u_{(1_i)_{(3)},(1_k)_{(2)}} \otimes u_{t,e} \\
        &= \sum_{t \in \onb(\mathcal{F}_0)} u_{1_i,(1_k)_{(1)}t(1_k)_{(2)}} \otimes u_{t,e} \\
        &= u_{1_i,1_k} \otimes d_k^{-2}u_{\sigma_k^{-2},e}
    \end{align*}
    Hence, applying $\id \otimes \varphi$, we find
    \begin{align*}
        u_{(1_i)_{(1)},(1_k)_{(1)}}u_{(1_i)_{(3)},(1_k)_{(2)}} \varphi(u_{(1_i)_{(2)},e}) = u_{1_i,1_k} d_k^{-2}\varphi(u_{\sigma_k^{-2},e}) \\
        u_{1_i,1_k} d_i^{-2}\varphi(u_{\sigma_i^{-2},e}) = u_{1_i,1_k} d_k^{-2}\varphi(u_{\sigma_k^{-2},e})
    \end{align*}
    Hence, when $i \approx k$, $u_{1_i,1_k} \neq 0$, meaning $d_i^{-2}\varphi(u_{\sigma_i^{-2},e}) = d_k^{-2}\varphi(u_{\sigma_k^{-2},e})$. The second claim follows from the fact that $\psi = \varphi \circ R$.
\end{proof}

\begin{definition} \label{def: conditional expectation}
    We define the following conditional expectation.
    \begin{align*}
        \mathfrak{e}: M \to Z(M) \cong \ell^{\infty}(I): x \mapsto \sum_{y \in \onb(\mathcal{F}_0)} yxy^* = \sum_{i \in I} \frac{\tr(x \sigma_i^{-1})}{d_i} 1_i
    \end{align*}
\end{definition}

\begin{lemma} \label{equivariant endomorphisms of M}
    For any $x \in M_0$, we have that $\mathfrak{e}\left[(\id \otimes \varphi)\alpha(x) \right]$ is a well-defined element of $\ell^{\infty}(I)$.
\end{lemma}
\begin{proof}
    By direct calculation, we get the following.
    \begin{align*}
        \mathfrak{e}[(\id \otimes \varphi)\alpha(x)] =& \sum_{y \in \onb(\mathcal{F}_0)} \mathfrak{e}(y) \varphi(u_{y,x}) \\
        =& \; \sum_{\substack{i \in I \\ y \in \onb(\mathcal{F}_0)}} \frac{\tr(y \sigma_i^{-1})}{d_i} 1_i \varphi(u_{y,x}) \\
        =& \sum_{i \in I} \frac{1}{d_i^2} 1_i \varphi(u_{\sigma_i^{-2},x}) \\
        =& \sum_{i \in I} 1_i \frac{\varphi(u_{\sigma_i^{-2},x})}{d_i^2}
    \end{align*}
    Now, as $x$ is supported on finitely many orbits, and $\frac{\varphi(u_{\sigma_i^{-2},x})}{d_i^2}$ is constant on orbits, it follows that this is indeed a well-defined element of $\ell^{\infty}(I)$.
\end{proof}

\begin{corollary}
    Note that $\mathfrak{e}[(\id \otimes \psi)\alpha^{\op}(x)] = \mathfrak{e}[(\id \otimes \varphi)\alpha(x)]^{\op} \in \ell^{\infty}(I)$ is now also a well-defined element of $\ell^{\infty}(I)$.
\end{corollary}

\section{Equivariant representation theory} \label{sc: equivariant representation theory}

 Let us recall from \cite{DeCommerDeRo2024} the notion of an $\alpha$-equivariant representation.
\begin{definition} \label{def: equivariant representation}
    An $\alpha$-equivariant representation of $\mathbb{G}$ on a Hilbert-$M$-bimodule $H$ with left- and right $M$-module structure $\lambda_H$ and $\rho_H$ respectively, is a representation $U \in B(H) \ovtimes L^{\infty}(\mathbb{G})$ such that
    \begin{align*}
        U (\lambda_H(x) \otimes 1) = [(\lambda_H \otimes 1) \alpha(x)]U \text{ and } [(\rho_H(x^{\op}) \otimes 1)]U = U(\rho_H \otimes 1) \alpha^{\op}(x^{\op}) \text{ for all } x \in M
    \end{align*}
    We say the representation is of finite type if $H$ is a finite type Hilbert-$M$-bimodule as in definition \ref{def: Hilbert-M-bimodule}.
\end{definition}
Note that for any $W \subset I$ which is invariant under $\approx$, and any $\alpha$-equivariant representation $U$ of $\mathbb{G}$ on a Hilbert-$M$-bimodule $H$, we get that $\lambda_H(1_W)$ and $\rho_H(1_W^{\op})$ are $M$-bimodular self-intertwiners of $U$.

\begin{lemma} \label{lem: U_0 is equivariant}
    The unitary implementation $U_0$ from definition \ref{def: unitary implementation} is an $\alpha$-equivariant representation in the sense of definition \ref{def: equivariant representation} with respect to the bimodule structure introduced in example \ref{ex: bimodule structure on F_n}.
\end{lemma}
\begin{proof}
    This follows more generally from \cite{Vaes01}, but nonetheless we include the calculation in our specific case. Denote for any $a,b \in M_0$ by $E_{a,b}$ the rank one operator $\xi \mapsto \langle \xi, \vect{b} \rangle \vect{a}$ in $B(L^2(M))$. We can calulate as follows, for arbitrary $x \in M$.
    \begin{align*}
        U_0(\lambda(x) \otimes 1)U_0^* =& \sum_{a,b,c,d \in \onb(\mathcal{F}_0)} E_{a,b} \lambda(x) E_{c,d} \otimes u_{a,b}u_{\mu(d)^*,c^*} \\
        =& \sum_{a,b,c,d \in \onb(\mathcal{F}_0} \del(b^*xc) E_{a,d} \otimes u_{a,b}u_{\mu(d)^*,c^*} \\
        =& \sum_{a,c,d \in \onb(\mathcal{F}_0} E_{a,d} \otimes u_{a,xc}u_{\mu(d)^*,c^*} \\
        =& \; \sum_{a,d \in \onb(\mathcal{F}_0)} E_{a,d} \otimes u_{a \mu(d)^*,x} \\
        =& \sum_{z \in \onb(\mathcal{F}_0)} \lambda(z) \otimes u_{z,x} \\
        =& (\lambda \otimes \id)\alpha(x)
    \end{align*}
    And also as follows.
    \begin{align*}
        U_0^*(\rho(x^{\op}) \otimes 1)U_0 =& \sum_{a,b,c,d \in \onb(\mathcal{F}_0)} E_{a,b}\rho(x^{\op})E_{c,d} \otimes u_{\mu(b)^*,a^*}u_{c,d} \\
        =& \sum_{a,b,c,d \in \onb(\mathcal{F}_0)} \del(b^*c \mu^{-1/2}(x)) E_{a,d} \otimes u_{\mu(b)^*,a^*}u_{c,d} \\
        =& \; \sum_{a,c,d \in \onb(\mathcal{F}_0)} E_{a,d} \otimes u_{\mu^{1/2}(x)^* \mu(c)^*,a^*} u_{c,d} \\
        =& \sum_{a,d \in \onb(\mathcal{F}_0)} E_{a,d} \otimes u_{\mu^{1/2}(x)^*,a^*d} \\
        =& \sum_{z \in \onb(\mathcal{F}_0)} \rho(z^{\op}) \otimes u_{\mu^{1/2}(x)^*,\mu^{1/2}(z)^*} \\
        =& (\rho \otimes \id) \alpha^{\op}(x^{\op})
    \end{align*}
    Hence $U_0$ is indeed a unitary $\alpha$-equivariant representation.
\end{proof}

The following lemma allows to take relative tensor products of equivariant representations. It is a special case of \cite[Definition 5.2]{DeCommerDeRo2024}.

\begin{lemma} \label{lem: relative tensor product of representations}
    Let $U$ and $V$ be two $\alpha$-equivariant representations of $\mathbb{G}$ on Hilbert-$M$-bimodules $H$ and $K$ respectively. Consider the projection $P: H \otimes K \to H \mtimes K$ where $P = \sum_{y \in \onb(\mathcal{F}_0)} (\rho_H(y^{\op}) \otimes \lambda_K(\mu^{1/2}(y^*))$. Then $(P \otimes 1)U_{13}V_{23}$ is an $\alpha$-equivariant representation of $\mathbb{G}$ on $H \mtimes K$. We will denote $U \odot V := (P \otimes 1)U_{13}V_{23}$.
\end{lemma}
\begin{proof}
    We start by noting the following.
    \begin{align*}
        (P \otimes 1)U_{13}V_{23} =& \; \sum_{y \in \onb(\mathcal{F}_0)} (\rho_H(y^{\op}) \otimes \lambda_K(\mu^{1/2}(y^*)) \otimes 1)U_{13}V_{23} \\
        =& \; \sum_{y \in \onb(\mathcal{F}_0)} U_{13}[(\rho_H \otimes 1)\alpha^{\op}(y^{\op})]_{13}(1 \otimes \lambda_K(\mu^{1/2}(y^*)) \otimes 1)V_{23} \\
        =& \; \sum_{y,z \in \onb(\mathcal{F}_0)} U_{13} (\rho_H(z^{\op}) \otimes \lambda_K(\mu^{1/2}(y^*)) \otimes u_{\mu^{1/2}(y^*),\mu^{1/2}(z^*)}) V_{23} \\
        =& \; \sum_{z \in \onb(\mathcal{F}_0)} U_{13} \left(\rho_H(z^{\op}) \otimes \left[ (\lambda \otimes \id)\alpha(\mu^{1/2}(z^*)) \right] \right) V_{23} \\
        =& U_{13}V_{23}(P \otimes 1)
    \end{align*}
    Using this, one readily checks that
    \begin{align*}
        (\id \otimes \id \otimes \Delta)(P \otimes 1)U_{13}V_{23} =& \; (P \otimes 1 \otimes 1)U_{13}U_{14}V_{23}V{24} \\
        =& \; (P \otimes 1 \otimes 1) U_{13}V_{23}(P \otimes 1 \otimes 1)U_{14}V_{24}
    \end{align*}
    and hence $(P \otimes 1)U_{13}V_{23}$ is indeed a representation of $\mathbb{G}$. $\alpha$-equivariance is immediate.
\end{proof}

\begin{corollary}
    $U_0$ acts as a monoidal unit with respect to the tensor product $\odot$ in the sense that for any $\alpha$-equivariant representation $U$ on a Hilbert-$M$-bimodule $H$, we have that $U \odot U_0$, $U$, and $U_0 \odot U$ are equivalent.
\end{corollary}
\begin{proof}
    Using lemmas \ref{lem: U_0 is equivariant} and \ref{lem: relative tensor product of representations}, the result follows easily by noting that the canonical unitary $M$-bimodular isomorphisms $L^2(M) \mtimes H \cong H \cong H \mtimes L^2(M)$ intertwine the representations.
\end{proof}

\begin{definition} \label{def: equivariant representation category}
    We define the $2$-category $\rep_f^{\alpha}(\mathbb{G})$ whose $0$-cells are the orbits of $\alpha$, whose $1$-cells are finite type unitary $\alpha$-equivariant representations of $\mathbb{G}$, and whose $2$-cells are $M$-bimodular bounded linear intertwiners of the representations. Composition of $1$-cells is given by the relative tensor product, and composition of $2$-cells by the ordinary composition of operators.
\end{definition}

The rest of this section will be dedicated to describing important properties of the category $\rep_f^{\alpha}(\mathbb{G})$. In particular, we will show that it abides by the conditions necessary for the reconstruction theorem \cite[Theorem 3.0.1]{Rollier2025}. To this end, the following lemma will be indispensible.

\begin{lemma} \label{lem: integrating intertwiners}
    Let $X$ and $Y$ be two $\alpha$-equivariant representations of $\mathbb{G}$ on Hilbert-$M$-bimodules $H$ and $K$ respectively. Let $T: K \to H$ be any bounded $M$-bimodular linear map such that there exists some finitely supported central projection $p \in M_0$ for which $T \circ \lambda_K(p) = T$. Then there exists a bounded $M$-bimodular linear map $\Phi(T)$, such that for any $i \in I$ we get
    \begin{align*}
        \lambda(1_i) \Phi(T) := (\id \otimes \varphi) \left[ \sum_{y \in \onb(1_i \cdot \mathcal{F}_0)} (\lambda_H(y) \otimes 1)X(T \otimes 1)Y^*(\lambda_K(y^*) \otimes 1) \right].
    \end{align*}
    This map satisfies
    \begin{align}
        X(\Phi(T) \otimes 1)Y^* = \Phi(T) \otimes 1. \label{eq: Phi(T) intertwines}
    \end{align}
\end{lemma}
\begin{proof}
    A direct calculation immediately shows that $\Phi(T)$ is $M$-bimodular. It remains to show boundedness. By an $M$-valued version of the Cauchy-Schwarz inequality, it suffices to show that
    \begin{align*}
        \sum_{y \in \onb(1_i \cdot \mathcal{F}_0)} (\id \otimes \varphi) \left[ (\lambda_H(y) \otimes 1)X(\lambda_H(p) \otimes 1)X^*(\lambda_H(y^*) \otimes 1) \right]
    \end{align*}
    is uniformly bounded in $i \in I$. One may note however that this equals
    \begin{align*}
        \sum_{y \in \onb(1_i \cdot \mathcal{F}_0)} & (\id \otimes \varphi) \left[ (\lambda_H \otimes \id)[(y \otimes 1)\alpha(p)] XX^* (\lambda_H \otimes \id)[(y \otimes 1)\alpha(p)]^*
        \right] \\
        \leq & \norm{XX^*}  \lambda_H \left[ 1_i (\mathfrak{e} \otimes \varphi)(\alpha(p)) \right]
    \end{align*}
    which is uniformly bounded by lemma \ref{equivariant endomorphisms of M}.

    Finally, we show \eqref{eq: Phi(T) intertwines}. To this end, define for any $y,z \in M$
    \begin{align*}
        S_{y,z} := (\id \otimes \varphi) \left[ (\lambda_H(y) \otimes 1)X(T \otimes 1)Y^*(\lambda_K(z^*) \otimes 1) \right]
    \end{align*}
    by definition, we then have 
    \begin{align*}
        \Phi(T) = \sum_{x \in \onb(1_i \cdot \mathcal{F}_0)} S_{x,x},
    \end{align*}
    with convergence in the strong operator topology.
    Then a simple calculation using left invariance of $\varphi$ shows that for any $x \in M_0$, we have
    \begin{align*}
        S_{y,z} & \otimes u_{y,x}u_{z,x}^* = (\id \otimes \id \otimes \varphi) \left[ (\lambda_H(y) \otimes u_{y,x} \otimes 1)X_{12}X_{13}(T \otimes 1 \otimes 1)Y^*_{13}Y^*_{12}(\lambda_K(z^*) \otimes u_{z,x}^* \otimes 1) \right].
    \end{align*}
    Summing over $x \in \onb(\mathcal{F}_0)$ and $y\in \onb(1_i \cdot \mathcal{F}_0)$ and $z \in \onb(1_j \cdot \mathcal{F}_0)$, the left hand side becomes $\delta_{i,j} \lambda_H(1_i)\Phi(T) \otimes 1$. If we sum over $y \in \onb(1_i \cdot \mathcal{F}_0)$ and $z \in \onb(1_j \cdot \mathcal{F}_0)$, the right hand side becomes
    \begin{align*}
        (\lambda_H(1_i) \otimes 1)(\id& \otimes \id \otimes \varphi) \left[ ((\lambda_H \otimes \id)\alpha(x) \otimes 1)X_{12}X_{13}(T \otimes 1 \otimes 1)Y^*_{13}Y^*_{12}((\lambda_K \otimes \id)\alpha(x^*) \otimes 1) \right] (\lambda_K(1_j) \otimes 1) \\
        =& (\lambda_H(1_i) \otimes 1) X \left[\left( (\id \otimes \varphi)(\lambda_H(x) \otimes 1)X(T \otimes 1)Y^*(\lambda_K(x^*) \otimes 1) \right) \otimes 1 \right]Y^* (\lambda_K(1_j) \otimes 1) \\
        =& \; (\lambda_H(1_i) \otimes 1)X(S_{x,x} \otimes 1)Y^*(\lambda_K(1_j) \otimes 1)
    \end{align*}
    Summing over $x \in \onb(\mathcal{F}_0)$, this then becomes $(\lambda_H(1_i) \otimes 1)X(\Phi(T) \otimes 1)Y^*(\lambda_K(1_j) \otimes 1)$. Now we can finish since
    \[
    \Phi(T) \otimes 1 = \sum_{i,j \in I} \delta_{i,j} \lambda_H(1_i) \Phi(T) \otimes 1 = \sum_{i,j \in I} (\lambda_H(1_i) \otimes 1)X(\Phi(T) \otimes 1)Y^*(\lambda_K(1_j) \otimes 1) = X(\Phi(T) \otimes 1)Y^*.
    \]
\end{proof}

\begin{remark}
    Note that by a very similar argument, we may also take any bounded $M$-bimodular $T: K \to H$ such that $T \circ \rho_K(p^{\op}) = T$ for some central projection $p \in M_0$, from which we can then define an $M$-bimodular bounded linear map $\Psi(T)$ such that for any $i \in I$
    \begin{align*}
        \rho_H(1_i)\Psi(T) := (\id \otimes \psi) \left[ \sum_{y \in \onb(\mathcal{F}_0 \cdot 1_i)} (\rho_H((y^*)^{\op}) \otimes 1) X^*(T \otimes 1)Y(\rho_K(y^{\op}) \otimes 1) \right].
    \end{align*}
    This map then satisfies
    \begin{align*}
        X^*(\Psi(T) \otimes 1)Y = \Psi(T) \otimes 1.
    \end{align*}
\end{remark}

\begin{lemma} \label{lem: equivalent to unitary}
    Every $\alpha$-equivariant representation of $\mathbb{G}$ is equivalent to a unitary $\alpha$-equivariant representation.
\end{lemma}
\begin{proof}
    Let $X$ be any $\alpha$-equivariant representation of $\mathbb{G}$ on a Hilbert-$M$-bimodule $H$. Then by invertibility and boundedness of $X$, there exist $m,s \in \mathbb{R}^+$ such that $m1 \leq XX^* \leq s1$. Hence, it follows that for any central $p \in M_0$, we get
    \begin{align*}
        m \cdot \lambda_H[(\mathfrak{e} \otimes \varphi)\alpha(p)] \leq \Phi(\lambda_H(p)) \leq s \cdot \lambda_H[(\mathfrak{e} \otimes \varphi)\alpha(p)]
    \end{align*}
    If we take $p = \frac{d_i^2}{\varphi(u_{\sigma_i^{-2},1_i})}1_i$, then $(\mathfrak{e} \otimes \phi)\alpha(p)$ is the indicator function of the equivalence class of $i$ under $\approx$. Choosing such a representative $x$ of every equivalence class, and summing all $\Phi(\lambda_H(x))$ we get in this way, we obtain a positive, bounded, invertible, $M$-bimodular linear map $T: H \to H$ which satisfies
    \begin{align*}
        T \otimes 1 = X(T \otimes 1)X^*
    \end{align*}
    Then clearly, $Y := (T^{-1/2} \otimes 1)X(T^{1/2} \otimes 1)$ is an equivalent unitary $\alpha$-equivariant representation of $\mathbb{G}$.    
\end{proof}

\begin{lemma} \label{lem: decomposition into irreducibles}
    Every finite type $\alpha$-equivariant representation of $\mathbb{G}$ is equivalent to a direct sum of irreducible representations.
\end{lemma}
\begin{proof}
    Let $X$ be any $\alpha$-equivariant representation of $\mathbb{G}$ on a finite type Hilbert-$M$-bimodule $H$. We may assume that $X$ is unitary, so that $\text{End}(X)$ is a $C^*$-algebra.
    Note first that since $\lambda_H(1_W)$ is an endomorphism of $X$ for any $W \subset I$ which is invariant under $\approx$, we may assume that there is an equivalence class $W$ of $\approx$ such that $\lambda_H(1_W) = 1$.
    We will show that the map $\text{End}(X) \to B(1_i \cdot H): T \mapsto \lambda_H(1_i)T$ is injective for every $i \in W$, and then by finite dimensionality of $\lambda_H(1_i)H$, the $C^*$-algebra $\text{End}(X)$ must be finite-dimensional, so that we can decompose by minimal projections.

    Take then $i \in W$ arbitrarily, and suppose that $\lambda_H(1_i)T = 0$ for some $T \in \text{End}(X)$. Then we can calculate as follows for any $j \in I$.
    \begin{align*}
        0 =& \sum_{y \in \onb(1_j \cdot \mathcal{F}_0)}(\id \otimes \varphi)[(\lambda_H(y) \otimes 1)X(\lambda_H(1_i)T\lambda_H(1_i) \otimes 1)X^*(\lambda_H(y^*) \otimes 1)] \\
        =& \sum_{y \in \onb(1_j \mathcal{F}_0)} (\id \otimes \varphi) \left[ (\lambda_H \otimes \id)[(y \otimes 1)\alpha(1_i)] (T \otimes 1) (\lambda_H \otimes \id)[\alpha(1_i)(y^* \otimes 1)] \right] \\
        =& \lambda_H[1_j (\mathfrak{e} \otimes \varphi)\alpha(1_i)] T \\
        =& \delta_{j \in W} \frac{\varphi(u_{\sigma_i^{-2},1_i})}{d_i^2} \lambda_H(1_j) T 
    \end{align*}
    Summing over all $j \in W$, it follows that $T = 0$, and hence the claim holds.
\end{proof}

\begin{lemma} \label{lem: decomposition into finite type}
    Every $\alpha$-equivariant representation of $\mathbb{G}$ decomposes into a direct sum of finite type $\alpha$-equivariant representations.
\end{lemma}
\begin{proof}
    Let $X$ be some $\alpha$-equivariant representation of $\mathbb{G}$ on a Hilbert-$M$-bimodule $H$. We may assume that $X$ is unitary, and that there are equivalence classes $W,W'$ of $\approx$ such that $\lambda_H(1_W) = \id_H = \rho_H(1_{W'}^{\op})$. Fix $i \in W$ and $j \in W'$ with $\lambda_H(1_i) \rho_H(1_j^{\op}) \neq 0$. Let $T \in \mathcal{K}(H)$ be any positive, nonzero, compact, $M$-bimodular operator which is bounded above by $\lambda_H(1_i) \rho_H(1_j^{\op})$. Then recall that by the definition of $\Phi$ (lemma \ref{lem: integrating intertwiners}) and $\alpha$-equivariance of $X$, for any $k,l \in I$, we get that
    \begin{align*}
        &\lambda_H(1_l)\rho_H(1_k^{\op}) \Phi(T) \\
        &= \sum_{y \in \onb(1_l \cdot \mathcal{F}_0)} (\id \otimes \varphi) \left[ (\lambda_H \otimes \id)[(y \otimes 1)\alpha(1_i)]X (T \otimes 1)[(\rho_H \otimes R)\alpha(1_k)]X^* (\lambda_H \otimes \id)[\alpha(1_i)(y^* \otimes 1)]\right] \\
        &= \sum_{\substack{x,y,z \in \onb(1_l \cdot \mathcal{F}_0) \\ t \in \onb(1_j \cdot \mathcal{F}_0)}} (\id \otimes \varphi) \left[ (\lambda_H(yx) \otimes u_{x,1_i}) X (T \rho_{H}(t^{\op}) \otimes R(u_{t,1_k})) X^* (\lambda_H(zy^*) \otimes u_{z,1_i}) \right]
    \end{align*}
    Now, as $X \in \mathcal{M}(\mathcal{K}(H) \otimes_r C_r(\mathbb{G}))$, we get that
    \[
    X \left( \sum_{t \in \onb(1_j \cdot \mathcal{F}_0)} T \rho_H(t^{\op}) \otimes R(u_{t,1_k}) \right) X^* \in \mathcal{K}(H) \otimes_r C_r(\mathbb{G}).
    \]
    Hence, $\lambda_H(1_l) \rho_H(1_k^{\op}) \Phi(T)$ is compact, and moreover, that 
    \[
    \norm{(\lambda_H(1_l) \otimes 1) X \left( \sum_{t \in \onb(1_j \cdot \mathcal{F}_0)} T \rho_H(t^{\op}) \otimes R(u_{t,1_k}) \right) X^*} \to 0 \text{ for } l \to \infty.
    \]
    Therefore,
    \[
    \sum_{l \in I} (\lambda_H(1_l) \otimes 1) X \left( \sum_{t \in \onb(1_j \cdot \mathcal{F}_0)} T \rho_H(t^{\op}) \otimes R(u_{t,1_k}) \right) X^*
    \]
    converges in norm to $X \left( \sum_{t \in \onb(1_j \cdot \mathcal{F}_0)} T \rho_H(t^{\op}) \otimes R(u_{t,1_k}) \right) X^*$, and finally, we get that
    \[
    \sum_{l \in I} \lambda_H(1_l) \rho_H(1_k^{\op}) \Phi(T) = \rho_H(1_k^{\op}) \Phi(T)
    \]
    is also a convergence in norm, so that $\rho_H(1_k^{\op}) \Phi(T)$ is compact.

    By faithfulness of $\id \otimes \varphi$ and positivity of $T$, $\Phi(T)$ is also nonzero, so we can use functional calculus of compact operators to find an intertwiner $P_k$ of $X$ for which $\rho_H(1_k^{\op})P_k$ is a nonzero finite rank projection. Since $C^*(\Phi(T)) \to B(H \cdot 1_k): Q \mapsto \rho_H(1_k^{\op})Q$ is an injective $*$-homomorphism as in the proof of lemma \ref{lem: decomposition into irreducibles}, this shows that $X$ admits an $\alpha$-equivariant subrepresentation $X_k$ on a Hilbert-$M$-subbimodule $H_k$ for which $\rho_{H_k}(1_k^{\op})$ is finite rank. 

    Now suppose $\rho_{H_k}(1_l^{\op})$ is not finite rank for some $l \in I$. Then we can inductively repeat the previous argument infinitely many times to find projections $P^1_l, P^2_l, \ldots$ which sum to $\id_{H_k}$, which intertwine $X_k$, and for which $\rho_{H_k}(1_l^{\op})P^n_l$ is finite rank for every $n \in \mathbb{N}$. However, again since $Q \mapsto \rho_{H_k}(1_k^{\op})Q$ is an injective $*$-homomorphism from the intertwiner space of $X_k$ into a finite dimensional $C^*$-algebra, we must get that all but finitely many of these $P^n_l$ are zero, which is in contradiction with the assumption that $\rho_{H_k}(1_l^{\op})$ is not finite rank. Hence, $H_k$ is a finite type Hilbert-$M$-subbimodule of $H$.
    
    By Zorn's lemma, $X$ now completely decomposes into finite type $\alpha$-equivariant subrepresentations.
\end{proof}

\begin{corollary} \label{cor: split into finite type irreducible unitaries}
    Together, lemmas \ref{lem: equivalent to unitary}, \ref{lem: decomposition into irreducibles}, and \ref{lem: decomposition into finite type} show that every $\alpha$-equivariant representation of $\mathbb{G}$ is equivalent to a direct sum of irreducible, unitary ones, which are necessarily finite type.
\end{corollary}

\begin{lemma} \label{lem: scaled representation}
    Let $U$ be any $\alpha$-equivariant representation of $\mathbb{G}$ on a Hilbert-$M$-bimodule $H$. For any $t \in \mathbb{R}$, we denote by $H_t$ the Hilbert-$M$-bimodule which as a Hilbert space is just $H$, but with the bimodule structure twisted such that
    \[
    \lambda_{H_t}(x) = \lambda_H(\mu^{-it}(x)) \text{ and } \rho_{H_t}(x^{\op}) = \rho_H(\mu^{-it}(x)^{\op}).
    \]
    Then $U_t := (\id \otimes \tau_t)U$ is an $\alpha$-equivariant representation of $\mathbb{G}$ on $H_t$. Moreover, for any $t \in \mathbb{R}$, we have that $U_t$ is equivalent to $U$.
\end{lemma}
\begin{proof}
    It is clear that $U_t$ is a representation. We show $\alpha$-equivariance as follows, using the fact that $\{\mu^{it}(y) | y \in \onb(\mathcal{F}_0)\}$ is of course an $\onb(\mathcal{F}_0)$ for any $t \in \mathbb{R}$.
    \begin{align*}
        U_t (\lambda_{H_t}(x) \otimes 1) &= (\id \otimes \tau_t) U (\lambda_H(\mu^{-it}(x)) \otimes 1) \\
        &=(\id \otimes \tau_t) [(\lambda_H \otimes \id)\alpha(\mu^{-it}(x))U] \\
        &= \sum_{y \in \onb(\mathcal{F}_0)} (\lambda_H(y) \otimes \tau_t(u_{y, \mu^{-it}(x)})) U_t \\
        &= \sum_{y \in \onb(\mathcal{F}_0)} (\lambda_H(\mu^{-it}(y)) \otimes u_{y,x}) U_t \\
        &= (\lambda_{H_t} \otimes \id)\alpha(x) U_t
    \end{align*}
    Right equivariance is shown similarly.

    To show that $U_t$ is equivalent to $U$, we suppose first that $U$ is unitary and irreducible. Fix then some $i \in I$, and recall that
    \begin{equation*}
        T_t := \sum_{y \in \onb(\mathcal{F}_0)}(\id \otimes \varphi) (\lambda_{H_t}(y) \otimes 1)U_t(\lambda_H(1_i) \otimes 1)U^*(\lambda_{H}(y^*) \otimes 1) \label{eq: U equivalent to U_t}
    \end{equation*}
    is a bounded $M$-bimodular intertwiner from $U$ to $U_t$. We will show that $T_t$ is nonzero for small values of $t$, whence it then follows that it must be a scalar multiple of a unitary by irreducibility of $U$. Fix $\xi \in H$ with $1_i \cdot \xi = \xi$. Consider the following map.
    \[
    \mathbb{R} \ni t \mapsto \omega_{\xi,\xi}(T_t) \in \mathbb{C}
    \]
    Note first that $\omega_{\xi,\xi}(T_0) \neq 0$, so it suffices to show continuity of this map. We calculate as follows.
    \begin{align*}
        \omega_{\xi,\xi}(T_t) &= \sum_{x,y,z \in \onb(1_i \cdot \mathcal{F}_0)} (\omega_{\xi,\xi} \otimes \varphi) \left[ (\lambda_{H_t}(yx) \otimes u_{x,1_i}) U_t U^* (\lambda_{H}(zy^*) \otimes u_{z,1_i}) \right] \\
        &= \sum_{x,y,z \in \onb(1_i \cdot \mathcal{F}_0)} \left\langle U_tU^* (zy^* \cdot \xi \otimes \vect{u_{z,1_i}}) , \mu^{-it}(x^*y^*) \cdot \xi \otimes \vect{u_{x,1_i}^*} \right\rangle
    \end{align*}
    Where $\vect{u_{x,1_i}}$ denotes $u_{x,1_i}$ seen as an element of $L^2(\mathbb{G})$. Now, by SOT-continuity of $t \mapsto \tau_t$, we also get that $t \mapsto U_tU^*$ is WOT-continuous, so the above is a finite sum of continuous functions, and hence we find a neighbourhood of $0$ for which $T_t \neq 0$ such that $U_t$ is equivalent to $U$. Now, suppose still not all $U_t$ are equivalent. Then let $t_0$ be the infimum of all $t > 0$ for which $U_t$ is not equivalent to $U$. Repeating this argument for $U_{t_0}$, we find a small neighbourhood of $t_0$ around which all $U_t$ are equivalent to $U_{t_0}$, but this is in contradiction with $t_0$ being an infimum.

    The result for arbitrary $U$ now follows from decomposing into irreducible unitary $\alpha$-equivariant subrepresentations.
\end{proof}

Fix now $U$, some $\alpha$-equivariant unitary representation of $\mathbb{G}$ on a Hilbert-$M$-bimodule $H$. By the previous proposition, we must find some strongly continuous one parameter group of unitaries $A_t \in B(H)$ such that $(A_{-t} \otimes 1)U(A_t \otimes 1) = U_t$. By Stone's theorem \cite{Stone1932}, we can then write $A_t = \exp(itB)$ for some unique possibly unbounded self-adjoint operator $B$. Writing $\Omega_U = \exp(B)$, we get a possibly unbounded positive operator $\Omega_U$ on $H$ such that $U_t = (\Omega_U^{-it} \otimes 1) U (\Omega_U^{it} \otimes 1)$.

\begin{definition} \label{def: Omega}
    For every unitary $\alpha$-equivariant representation $U$ of $\mathbb{G}$ on a Hilbert-$M$-bimodule $H$, we let $\Omega_U$ be the unbounded operator on $H$ such that $U_t = (\Omega_U^{-it} \otimes 1) U (\Omega_U^{it} \otimes 1)$, where $U_t$ is as in \ref{lem: scaled representation}. By decomposing into irreducible unitary representations, this $\Omega_U$ exists and is unique for every $U$.
\end{definition}
One may note now that because $\Omega_{U}^{it}$ must be $M$-bimodular with respect to the bimodule structures on $H$ and $H_t$, we must have $\Omega_U(x \cdot \xi \cdot y) = \mu(x) \cdot \Omega_U(\xi) \cdot \mu(y)$ for any $\xi$ in the domain of $\Omega_U$ and any $x,y \in M_0$. This goes to show that $\Omega_U$ restricts to $1_i \cdot H \cdot 1_j$ for any $i,j \in I$. In particular, by positivity we must then have that whenever $H$ is finite type, we can find a basis of eigenvectors of $\Omega_U$ in $H_0 := M_0 \cdot H \cdot M_0$. Let $\xi,\eta$ be such eigenvectors, and denote $\Uelt{\xi}{\eta} := (\omega_{\xi,\eta} \otimes \id)U$. Then we must have $\tau_t(\Uelt{\xi}{\eta}) = \Uelt{\Omega_U^{it}(\xi)}{\Omega_U^{it}(\eta)}$. It follows that $\tau_t$ is analytic on $\Uelt{\xi}{\eta}$, and that $\tau_{-i/2}(\Uelt{\xi}{\eta}) = \Uelt{\Omega_U^{-1/2}(\xi)}{\Omega_U^{1/2}(\eta)}$. By linearity, this formula generalises to any $\xi,\eta \in H_0$. Using also the fact that $S(\Uelt{\xi}{\eta}) = \Uelt{\eta}{\xi}^*$, we now find that $R(\Uelt{\xi}{\eta}) = \Uelt{\Omega_U^{-1/2}(\eta)}{\Omega_U^{1/2}(\xi)}^*$.

\begin{definition} \label{def: conjugate representation}
    Let $U$ be any finite type $\alpha$-equivariant representation of $\mathbb{G}$ on a Hilbert-$M$-bimodule $H$. Denote by $\overline{H}$ the dual space to $H$ with the standard $M$-bimodule structure given by $x \cdot \overline{\xi} \cdot y := \overline{y^* \cdot \xi \cdot x^*}$, and by $J: B(H) \to B(\overline{H})$ the anti-isomorphism such that $J(T) \overline{\xi} = \overline{T^* \xi}$. Then denote $\overline{U} := (J \otimes R)U$. Then $\overline{U}$ is an $\alpha$-equivariant representation of $\mathbb{G}$ on $\overline{H}$, which we call the conjugate representation of $U$.
\end{definition}
\begin{proof}[Proof that definition \ref{def: conjugate representation} is well-defined]
    Since $J \otimes R$ is an anti-isomorphism, and $(R \otimes R) \circ \Delta = \chi \circ \Delta \circ R$, where $\chi$ denotes the flip map $a \otimes b \mapsto b \otimes a$, we automatically get that $\overline{U}$ is a representation. Now, $\alpha$-equivariance follows immediately from applying $(J \otimes R)$ to the equations establishing $\alpha$-equivariance of $U$.
\end{proof}

\begin{definition} \label{def: solution to conjugate equations}
    Let $U$ be any unitary $\alpha$-equivariant representation of $\mathbb{G}$ on a finite type Hilbert-$M$-bimodule $H$. We define the following operators.
    \begin{align*}
        &s_U: \mathcal{F}_0 \to H \mtimes \overline{H}: \vect{x} \mapsto \sum_{\xi \in \onb(H_0)} x \cdot \xi \otimes \Omega_{\overline{U}}^{1/2}(\overline{\xi}) \\
        &t_U: \mathcal{F}_0 \to \overline{H} \mtimes H: \vect{x} \mapsto \sum_{\xi \in \onb(H_0)} x \cdot \Omega_{\overline{U}}^{-1/2}(\overline{\xi}) \otimes \xi
    \end{align*}
    Note that for any $x \in M_0$, the sums defining $s_U(\vect{x})$ and $t_U(\vect{x})$ are finite, and therefore well-defined. As this requires $x \in M_0$, a priori we only have well-definedness on $\mathcal{F}_0$. However, we will see in corollary \ref{cor: conjugate equations have bounded solution} that very often $s_U$ and $t_U$ are bounded, and can therefore be extended to $L^2(M)$.
\end{definition}

\begin{lemma} \label{lem: conjugate equations for unitary finite type}
    Consider $U$, some unitary $\alpha$-equivariant representation of $\mathbb{G}$ on a finite type Hilbert-$M$-bimodule $H$. Let $\overline{U}$ be its conjugate representation as in \ref{def: conjugate representation}. Then $s_U$ and $t_U$ as in definition \ref{def: solution to conjugate equations} intertwine $U_0$ and $U \odot \overline{U}$ resp. $\overline{U} \odot U$.
\end{lemma}
\begin{proof}
    We show the claim for $s_U$ first. Note that since $\tau_t$ and $R$ commute, a simple calculation shows that $\Omega_{\overline{U}} = J(\Omega_U^{-1})$. Now, fix some $x \in M_0$ and $\eta,\eta' \in H_0$. Also denote $\Uelt{\xi}{\eta}$ for $(\omega_{\xi,\eta} \otimes \id)U$.
    \begin{align*}
        (\overline{\eta \otimes \overline{\eta'}} \otimes 1)(U \odot \overline{U})(s_U(\vect{x}) \otimes 1) =& \; (\overline{\eta \otimes \overline{\eta'}} \otimes 1)U_{13} \overline{U}_{23} \left( \sum_{\xi \in \onb(H_0)} x \cdot \xi \otimes \Omega_{\overline{U}}^{1/2}(\overline{\xi}) \otimes 1 \right) \\
        =& \; \sum_{\xi \in \onb(H_0)} \Uelt{\eta}{x \cdot \xi} \overline{U} \left( \substack{ \overline{\eta'} \\ \Omega_{\overline{U}}^{1/2}(\overline{\xi}) } \right) \\
        =& \; \sum_{\xi \in \onb(H)} \Uelt{\eta}{x \cdot \xi} R\left( \Uelt{ \Omega_{ U }^{ -1/2 } ( \xi  ) }{ \eta' } \right) \\
        =& \; \sum_{\substack{\xi \in \onb(H_0) \\ y \in \onb(\mathcal{F}_0)}} u_{y,x} \Uelt{y^* \cdot \eta}{\xi} \Uelt{\Omega_U^{-1/2}(\eta')}{\xi}^* \\
        =& \; \sum_{y \in \onb(\mathcal{F}_0)} \langle \Omega_U^{-1/2}(\eta') , y^* \cdot \eta \rangle u_{y,x} \\
        =& \; \left( \overline{\eta \otimes \overline{\eta'}} \otimes 1 \right) (s_U \otimes 1)U_0 (\vect{x} \otimes 1)
    \end{align*}
    The claim for $t_U$ can either be handled analogously, or seen to follow from the one for $s_U$ by the fact that $\overline{\overline{U}} = U$ and $t_U = s_{\overline{U}}$.
\end{proof}

\begin{corollary} \label{cor: conjugate equations have bounded solution}
    Let $U$ be any finite type $\alpha$-equivariant representation of $\mathbb{G}$ on a Hilbert-$M$-bimodule $H$, and suppose that there are finitely many orbits $W \in \mathcal{E}$ such that $\lambda_H(1_W) \neq 0$. Then $s_U$ as defined in \ref{def: solution to conjugate equations} is a bounded operator, and hence extends to $L^2(M) \to H \mtimes \overline{H}$. Similarly, if there are finitely many orbits $W \in \mathcal{E}$ such that $\rho_H(1_W^{\op}) \neq 0$, $t_U$ is bounded. In particular, $s_U$ and $t_U$ are bounded whenever $U$ is a finite direct sum of irreducible $\alpha$-equivariant representations.
\end{corollary}
\begin{proof}
    We only show the first claim, since the second is handled analogously. By decomposing into a direct sum, we may assume that there is a unique $W \in \mathcal{E}$ such that $\lambda_H(1_W) = 1$. Now note that for any $i \in W$, we get by direct calculation $s_U^* s_U \vect{1_i} = \del(1_i)^{-1} \tr(\lambda_H(1_i) \Omega_U) \vect{1_i}$. Hence,
    \[
    \lambda_H(1_i) \left( s_U^* s_U - \del(1_i)^{-1} \tr(\lambda_H(1_i) \Omega_U) \right) = 0.
    \]
    Now, repeating the argument from lemma \ref{lem: decomposition into irreducibles}, we find that $s_U^* s_U - \del(1_i)^{-1} \tr(\lambda_H(1_i) \Omega_U) = 0$, and hence $\norm{s_U}^2 = \del(1_i)^{-1} \tr(\lambda_H(1_i) \Omega_U)$, and this value does not depend on the choice of $i \in W$.
\end{proof}

Together, corollaries \ref{cor: split into finite type irreducible unitaries}, \ref{cor: conjugate equations have bounded solution} and lemma \ref{lem: conjugate equations for unitary finite type} show that $\rep_f^{\alpha}(\mathbb{G})$ yields a concrete unitary $2$-category of finite type Hilbert-$M$-bimodules in the sense of \cite[Definition 2.1.13]{Rollier2025}, which abides by the conditions of \cite[Theorem 3.0.1]{Rollier2025}. Before moving on, we show the following final proposition of this section.

\begin{proposition} \label{prop: corep into equivariant corep}
    Let $X$ be any representation of $\mathbb{G}$ on a Hilbert space $H$. Endow $L^2(M) \otimes H \otimes L^2(M)$ with $M$-bimodule structure given by the left- and right module structure of $M$ on the left- and right leg of the tensor product respectively. Then the representation $(U_0)_{1,4}X_{2,4}(U_0)_{3,4}$ of $\mathbb{G}$ on $L^2(M) \otimes H \otimes L^2(M)$ is $\alpha$-equivariant.
\end{proposition}
\begin{proof}
    Take $x \in M$ arbitrarily, and calculate as follows.
    \begin{align*}
        (U_0)_{1,4}X_{2,4}(U_0)_{3,4}(\lambda_{L^2(M) \otimes H \otimes L^2(M)}(x) \otimes 1) =& (U_0)_{1,4}X_{2,4}(U_0)_{3,4}(\lambda_{L^2(M)}(x) \otimes 1 \otimes 1 \otimes 1) \\
        =& ((U_0) (\lambda_{L^2(M)}(x) \otimes 1))_{1,4}X_{2,4}(U_0)_{3,4} \\
        =& ([(\lambda_{L^2(M)} \otimes \id)\alpha(x)](U_0))_{1,4}X_{2,4}(U_0)_{3,4} \\
        =& [(\lambda_{L^2(M) \otimes H \otimes L^2(M)} \otimes \id)\alpha(x)] (U_0)_{1,4}X_{2,4}(U_0)_{3,4}
    \end{align*}
    Right equivariance can be shown similarly.
\end{proof}

\section{$\mathbb{G}$ as an algebraic quantum group} \label{sc: algebraic quantum group}

\begin{definition} \label{def: O(G)}
    We define $\mathcal{O}(\mathbb{G})$ as the linear span of all $\Uelt{\xi}{\eta}$ where $U$ ranges over $\alpha$-equivariant unitary representations of $\mathbb{G}$ on finite type Hilbert-$M$-bimodules $H$, and $\xi,\eta$ are elements of $H_0 := M_0 \cdot H \cdot M_0$.
\end{definition}

\begin{theorem} \label{thm: O(G) is dense *-subalgebra}
    $\mathcal{O}(\mathbb{G})$ is a strongly dense $*$-subalgebra of $L^{\infty}(\mathbb{G})$. Moreover, $\Delta$ restricts to a coproduct $\mathcal{O}(\mathbb{G}) \to \mathcal{M}(\mathcal{O}(\mathbb{G}) \otimes_{\text{alg}} \mathcal{O}(\mathbb{G}))$, and the pair $(\mathcal{O}(\mathbb{G}), \Delta)$ is an algebraic quantum group in the sense of \cite{KustermansVanDaele97, VanDaele98}, with invariant functionals given by the restriction of $\varphi,\psi$.
\end{theorem}
\begin{proof}
    Take $\Uelt{\xi}{\eta}$ and $\Velt{\xi'}{\eta'}$ in $\mathcal{O}(\mathbb{G})$ arbitrarily. Then $\Uelt{\xi}{\eta}\Velt{\xi'}{\eta'} = (U_{13}V_{23}) \left( \substack{\xi \otimes \xi' \\ \eta \otimes \eta'} \right)$. Take $p,q,r,s \in M_0$ central projections such that $\xi \cdot p = \xi$, $q \cdot \xi' = \xi'$, $\eta \cdot r = \eta$, and $s \cdot \eta' = \eta'$. These exist by definition of $\mathcal{O}(\mathbb{G})$. Then by lemma \ref{lem: decomposition into finite type}, we can decompose $(U_0)_{13}(U_0)_{23}$ into finite type subrepresentations by $M$-bimodular endomorphisms. Hence, since $p \cdot L^2(M^2) \cdot q$ and $r \cdot L^2(M^2) \cdot s$ are finite-dimensional spaces, we can find a finite type subrepresentation $Q$ of $(U_0)_{13}(U_0)_{23}$ such that
    \begin{equation}
        (U_{13}V_{23}) \left( \substack{\xi \otimes \xi' \\ \eta \otimes \eta'} \right) = (U \odot Q \odot V) \left( \substack{ \xi \mtimes \vect{p \otimes q} \mtimes \xi' \\ \eta \mtimes \vect{r \otimes s} \mtimes \eta'  } \right) \in \mathcal{O}(\mathbb{G}), \label{eq: O(G) closed under product}
    \end{equation}
    and hence, $\mathcal{O}(\mathbb{G})$ is indeed an algebra. Next, since $\Uelt{\xi}{\eta}^* = \overline{U} \left( \substack{ \overline{\Omega_U^{1/2}(\eta)} \\ \overline{\Omega_U^{-1/2}(\xi)} } \right)$, we get that $\mathcal{O}(\mathbb{G})$ is closed under taking adjoints. We have shown $\mathcal{O}(\mathbb{G})$ to be a $*$-algebra. Applying proposition \ref{prop: corep into equivariant corep} to the left regular representation, and splitting into finite type subrepresentations by lemma \ref{lem: decomposition into finite type}, we get strong density of $\mathcal{O}(\mathbb{G})$ in $L^{\infty}(\mathbb{G})$.

    By the explicit calculation \eqref{eq: O(G) closed under product}, one  finds that $\Uelt{\xi}{\eta} \Velt{\xi'}{\eta'} = 0$ whenever there exist $p,q,r,s \in M_0$ central projections such that firstly $\xi \cdot p = \xi$, $q \cdot \xi' = \xi'$, $\eta \cdot r = \eta$, and $s \cdot \eta' = \eta'$ and secondly there is some finite type $M$-bimodular projection $Q$ in $\End((U_0)_{13}(U_0)_{23})$ for which $\vect{p \otimes q}$ lies in the range of $Q$ and $\vect{r \otimes s}$ lies in the orthogonal complement of this range or vice versa.  Since $(U_0)_{13}(U_0)_{23}$ splits into finite type $\alpha$-equivariant subrepresentations, it follows that $(\mathcal{O}(\mathbb{G}) \otimes 1) \Delta(\mathcal{O}(\mathbb{G}))$ and $(1 \otimes \mathcal{O}(\mathbb{G})) \Delta(\mathcal{O}(\mathbb{G}))$ are contained in $\mathcal{O}(\mathbb{G}) \otimes_{\text{alg}} \mathcal{O}(\mathbb{G})$. Clearly the antipode of $\mathbb{G}$ restricted to $\mathcal{O}(\mathbb{G})$ and the counit given by $\varepsilon \left( \Uelt{\xi}{\eta} \right) = \langle \eta, \xi \rangle$ now turn $(\mathcal{O}(\mathbb{G}), \Delta)$ into a multiplier Hopf-$*$-algebra as in \cite{VanDaele94}.

    It remains to show that the algebra $\mathcal{O}(\mathbb{G})$ lies in the domain of the Haar weights. To this end, fix some $i \in I$, and some $\xi \in H_0$ with $U$ some $\alpha$-equivariant representation of $\mathbb{G}$ on $H$. Then we get that for all but finitely many $j \in I$ $u_{1_i,1_j} \Uelt{\xi}{\xi} = 0$ and $\Uelt{\xi}{\xi} u_{1_i,1_j} = 0$. Therefore, there exists some central projection $p \in M_0$ such that 
    \[
    \del(1_i)^2 \Uelt{\xi}{\xi} = u_{1_i,p} \Uelt{\xi}{\xi} u_{1_i,p}.
    \]
    Now, $u_{1_i,p} \Uelt{\xi}{\xi} u_{1_i,p}$ must lie in the domain of $\varphi$ since $\varphi(u_{1_i,p}^2) < \infty$. Hence, $\Uelt{\xi}{\xi}$ does as well. Now the general statement follows from the fact that
    \[
    \Uelt{\xi}{\eta} = \frac{1}{4} \left( \Uelt{\xi + \eta}{\xi + \eta} - \Uelt{\xi - \eta}{\xi - \eta} + i \Uelt{\xi - i \eta}{\xi - i \eta} - i \Uelt{\xi + i \eta}{\xi + i \eta} \right)
    \]
    and hence $\Uelt{\xi}{\eta}$ must lie in the domain of $\varphi$ for any $\xi,\eta \in H_0$.
\end{proof}

\begin{corollary} \label{cor: uniqueness for equivariant tannaka-krein}
    We may now strengthen the conclusion of \cite[Theorem 3.0.1]{Rollier2025} as follows. Given any discrete quantum space $(M,\del)$, and any two locally compact quantum groups $\mathbb{G}_1, \mathbb{G}_2$ admitting proper actions $\alpha_1,\alpha_2$ on $(M,\del)$ respectively, the following are equivalent.
    \begin{enumerate}
        \item There exists an isomorphism of quantum groups $\pi: L^{\infty}(\mathbb{G}_1) \to L^{\infty}(\mathbb{G}_2)$ such that $(\id \otimes \pi) \circ \alpha_1 = \alpha_2$.
        \item There exists an equivalence of unitary $2$-categories $F: \rep_f^{\alpha_1}(\mathbb{G}_1) \to \rep_f^{\alpha_2}(\mathbb{G}_2)$ and $M$-bimodular unitary isomorphisms $\pi_U: H_U \to H_{F(U)}$, where $H_U$ denotes the Hilbert-$M$-bimodule underlying the representation $U$, such that for any $T \in \mor_{\rep_f^{\alpha_1}(\mathbb{G}_1)}(U,V)$ we get $F(T) = \pi_V \circ T \circ \pi_U^*$.
    \end{enumerate}
    Hence, we may now rightfully speak of an equivariant Tannaka-Krein \underline{\textbf{duality}}.
\end{corollary}
\begin{proof}
    Recall from the reconstruction \cite[Theorem 3.0.1]{Rollier2025} the algebraic quantum group $(\mathcal{A},\Delta)$ with an action on $(M,\del)$ such that its equivariant representation category is $\rep_f^{\alpha}(\mathbb{G})$. The algebra $\mathcal{A}$ was the universal $*$-algebra generated by $\Aelt{U}{\xi}{\eta}$ where $U$ ranges over unitary finite type $\alpha$-equivariant representations in $\rep_f^{\alpha}(\mathbb{G})$, and $\xi,\eta$ are finitely supported vectors in the finite type Hilbert-$M$-bimodules associated to those representations, under the relations that $A_U$ is unitary for every $U$, and that the intertwiners of $\text{Mor}_{\rep_f^{\alpha}(\mathbb{G})}(U,V)$ do indeed intertwine $A_U$ and $A_V$. In particular, there exists a surjective $*$-homomorphism $\mathcal{A} \to \mathcal{O}(\mathbb{G})$ which respects the comultiplications. Suppose now we are given a finite type, unitary $\alpha$-equivariant representation $U$ of $\mathbb{G}$, and vectors $\xi,\eta \in H_0 := M_0 \cdot H \cdot M_0$ such that $\Uelt{\xi}{\eta} = 0$. Take $e \in M_0$ arbitrarily with $\del(e^*e) = 1$, and consider
    \[
    T := \sum_{x,y \in \onb(\mathcal{F}_0)} E_{x \cdot \vect{e} \otimes \eta \otimes \vect{e^*} \cdot y^*,x \cdot \vect{e} \otimes \eta \otimes \vect{e^*} \cdot y^*} \in B(L^2(M) \otimes H \otimes L^2(M)).
    \]
    Let $C := (U_0)_{14}U_{24}(U_0)_{34}$, and $\Phi(T)$ be given by
    \[
    \sum_{z \in \onb(\mathcal{F}_0)} (\id \otimes \varphi) (\lambda_{L^2(M) \otimes H \otimes L^2(M)}(z) \otimes 1) C (T \otimes 1) C^* (\lambda_{L^2(M) \otimes H \otimes L^2(M)}(z^*) \otimes 1)
    \]
    as in lemma \ref{lem: integrating intertwiners}. A simple calculation then shows that $\langle \Phi(T) (\vect{e} \otimes \eta \otimes \vect{e^*}), \vect{e} \otimes \eta \otimes \vect{e^*} \rangle$ is nonzero, while $ \langle \Phi(T) (\vect{r} \otimes \xi \otimes \vect{s}), \vect{r} \otimes \xi \otimes \vect{s} \rangle = 0$ for every $r,s \in M_0$. Using \cite[Proposition 3.3.8]{Rollier2025}, it then follows that $\Aelt{H}{\xi}{\eta} = 0$, such that the $*$-homomorphism $\mathcal{A} \to \mathcal{O}(\mathbb{G})$ is in fact an isomorphism. It follows from \cite[Theorem 3.7]{VanDaele98} that the respective Haar weights of these algebraic quantum groups are also related through this isomorphism.
\end{proof}

\printbibliography[title=Bibliography]

\end{document}